\documentclass[11pt]{article}
\usepackage{amsmath}
\usepackage{amsthm}
\usepackage{amssymb}

\author{Aditya Joshi\\
School of Information and Physical Sciences\\
University of Newcastle\\
Callaghan, NSW 2308\\
Australia\\
aditya.joshi8342@gmail.com\\ 
 }

\date{\today}

 \usepackage{tikz}
\usepackage{float}
\usepackage{graphicx}

\title{On the induced neighbourhood of vertex-transitive graphs}

\usepackage{mathtools}

\usepackage{amsthm}

\usepackage{xcolor}

\newtheorem{theorem}{Theorem}[section]
\newtheorem{corollary}{Corollary}[theorem]
\newtheorem{lemma}[theorem]{Lemma}
\newtheorem{prop}[theorem]{Proposition}
\newtheorem{definition}[theorem]{Definition}

\begin{document}
\maketitle

\begin{abstract}
We present new conditions that eliminate a large class of graphs from being induced neighbourhoods of finite vertex-transitive graphs.
\end{abstract}

\section{Introduction}

In this paper graphs have neither loops nor multiple edges. Edges between two adjacent vertices $u$ and $v$ will be denoted by $[u,v]$. The number of edges incident with any vertex $v$ will be called the valency of $v$. A vertex with complete valency is one that is adjacent to all the remaining vertices of the graph.

\begin{definition}
\normalfont Let $X$ denote a graph and let $v$ denote some vertex of $X$. The \textit{induced neighbourhood} of $v$ is the subgraph induced on the set of vertices adjacent to $v$. We denote the set of vertices in the neighbourhood of $v$ by $N(v,X)$ and the induced neighbourhood of $v$ by $\langle N(v,X) \rangle$.
\end{definition}

\textit{Uniform neighbourhood graphs} are graphs for which all the induced neighbourhoods are isomorphic. In the literature these have been called constant-link graphs, constant-neighbourhood graphs, locally $X$-graphs and locally homogeneous graphs.  
Much of what has been studied of such graphs has been motivated by two problems posed by Zykov \cite[p.164]{Zykov}. The first of these problems asks to classify the graphs which are the induced neighbourhoods of some uniform neighbourhood graph. The second problem asks to classify the graphs which are induced neighbourhoods only for infinite uniform neighbourhood graphs. The two problems have turned out to be quite difficult. However, there are partial results on Zykov's problems that are largely devoted to studying specific families of graphs which may be induced neighbourhoods of some uniform neighbourhood graphs, determining necessary conditions such graphs must satisfy and examining sufficient conditions. See \cite{Survey, small_degree} and \cite[p.8]{partial_results_ref} for more information about the many partial results in the literature. \newline

It is obvious that vertex-transitive graphs are uniform neighbourhood graphs. One problem suggested in \cite{which_trees} is to find the smallest order for a connected uniform neighbourhood graph that is not vertex-transitive. This article manages to show that the order is bounded above by 16 by direct construction. This problem is completely resolved in  \cite{leastnon_vt} where it is shown that the minimum order is 10. 
To the author's knowledge the only work that has been done on classifying the induced neighbourhoods of vertex-transitive graphs is the paper \cite{t_realizations}.

The general methods used for showing that some graph is not the induced neighbourhood of a finite vertex-transitive graph include those used for the induced neighbourhood of uniform neighbourhood graphs. Another method is by contradicting the vertex or edge connectivity property of finite vertex-transitive graphs. We give an example of the latter method. It is well known that the edge-connectivity of a connected vertex-transitive graph is equal to the valency $k$ of the graph, and its vertex-connectivity is at least $\frac{2}{3}(k+1)$. Let $X$ denote a graph with $n$ vertices and $m$ edges. Suppose also that $X$ is not a clique. Using the connectivity results it is easy to show that if $m > \frac{n(n-2)}{2}$ or if the number of complete valency vertices in $X$ is greater than $\frac{1}{3}(n-2)$, then $X$ is not the induced neighbourhood of a finite vertex-transitive graph.

In this paper we give results that help to eliminate several other graphs as being induced neighbourhoods of finite vertex-transitive graph. We do this by showing that some substructure of this class of graphs violates an imposed condition on their automorphism groups. We also give examples of several graphs that satisfy this theorem.  

We now present some definitions that will be used in the subsequent sections.  

\begin{definition}[Regular group] 
\normalfont Let $G$ denote a permutation group acting transitively on a set $\Omega$. Then $G$ is called \textit{regular} if the stabiliser subgroup for each element is trivial. 
\end{definition}

\begin{definition}[Semiregular permutation]
\normalfont A permutation is \textit{semiregular} if the length of the cycles are the same in the cyclic decomposition of the permutation.
\end{definition}

We now give some equivalent definitions for a regular group. 

\begin{theorem}\label{1.4}
The following four conditions are equivalent for a transitive permutation group $G$ acting on a finite set $\Omega$.
\begin{enumerate}
\item The stabiliser of each element is the identity;
\item $|G| = |\Omega|$; 
\item For any $x,y \in \Omega$, there is a unique $f\in G$ such that $f(x)=y$; and 
\item All permutations in $G$ are semiregular. 
\end{enumerate}
\end{theorem}

\section{Results}
In this section we present some results which help to eliminate many graphs from being the induced neighbourhood of a finite vertex-transitive graph. The general approach taken is to restrict semiregular automorphisms on subsets of the induced neighbourhood of a finite vertex-transitive graph, by imposing conditions on the induced neighbourhood graph's structure that force it to do so. We then impose further conditions on these subsets such that their induced subgraph contradicts having such a semiregular automorphism. 
An {\it asymmetric graph} is a graph whose only automorphism is the identity permutation. 

The following lemma is needed and the proof of it is included for completeness.
\begin{lemma}\label{2.1}
Let $X$ denote an asymmetric graph and suppose that it is the induced neighbourhood of some finite vertex-transitive graph $Y$. Then $Y$ is a Cayley graph and $Aut(Y)$ is regular.
\end{lemma}
\begin{proof}
Let $X$ and $Y$ be as hypothesised. Let $u$ be an arbitrary vertex of $Y$. The stabiliser of $u$ also fixes every vertex in $N(u,Y)$ because $\langle N(u,Y) \rangle$ is asymmetric as it is isomorphic to $X$. This holds true for every vertex $u$ and as $Y$ is connected, we see that the stabiliser of $Y$ is the identity. Hence, by Theorem \ref{1.4} $Aut(Y)$ is a regular permutation group and every automorphism of $Y$ is semiregular. In particular, this means that non-identity automorphisms of $Y$ have no fixed points. By Sabidussi's classical theorem \cite{Sabd}, $Y$ is a Cayley graph.  
\end{proof}

For convenience we define an equivalence relation $\sim$ on the vertices of a graph X by letting $u\sim v$ if $\langle N(u,X) \rangle \cong \langle N(v,X) \rangle$. We denote the equivalence class containing a vertex $v$ by $[v]$. 
The following lemma generalises a result in \cite{t_realizations}.

\begin{lemma}\label{2.2}
Let $X$ denote a graph and let $[v]$ be of odd cardinality for some vertex $v\in X$. If $\langle N(v,X) \rangle $ does not have an automorphism that is the product of transpositions, then $X$ is not the induced neighbourhood of a finite Cayley graph.
\end{lemma}
\begin{proof}
Let $X$ be as hypothesised and suppose $Y$ is a connected finite Cayley graph for which $X$ is the common induced neighbourhood. Let $u$ be some vertex of $Y$. 
Let $v$ denote a vertex in $N(u,Y)$ which is in the equivalence class $[v]$ as hypothesised. We will also refer to $\langle N(u,Y) \rangle$ as $X$ for the rest of the proof. 

By Sabidussi's theorem $Aut(Y)$ contains a regular subgroup which we will denote by $R$.
Consider the set $A$ containing all automorphisms in $R$ mapping $u$ to a vertex in $[v]$. By Theorem \ref{1.4} we have that each such automorphism is unique and hence there is a bijection between $A$ and $[v]$. This implies that $A$ has odd cardinality since $[v]$ also does. 

Let $\alpha$ be the automorphism mapping $u$ to $v$ in $N(u,Y)$.  Now we look at
$\alpha^{-1}$.  The edge $[u,v]$ must get mapped to an edge.  Because $\alpha^{-1}$
sends $v$ to $u$, it must send $u$ to a neighbour of $u$, that is, $\alpha^{-1}(u)$ belongs to
$N(u,Y)$. For each automorphism $\alpha$ that maps $u$ to any $w\in [v]$ we have that $\alpha^{-1}(u)$ must also be in $[v]$ since $\langle N(w,Y) \cap N(u,Y)\rangle \cong \langle N(u,Y) \cap N(\alpha^{-1}(u),Y) \rangle$. Hence $A$ is closed under inverses. Thus, elements of $A$ and their inverses are paired and since $A$ has odd cardinality, there is an element $\alpha^{'}$ such that $\alpha^{'} = (\alpha^{'})^{-1}$, that is, $\alpha^{'}$ is an involution. 

Let $v^{'}\in[v]$ be such that $\alpha^{'}(u) = v^{'}$. Using Theorem \ref{1.4} we have that $\alpha^{'}$ must be a product of transpositions. 
Since $\alpha^{'}$ is an involution we also have that it maps $N(u,Y) \cap N(v^{'},Y)$ to itself. Hence we may restrict $\alpha^{'}$ to the graph $\langle N(u,Y) \cap N(v^{'},Y) \rangle$. This implies that the graph $\langle N(u,Y) \cap N(v^{'},Y) \rangle$ has an automorphism that is the product of transpositions, which leads to a contradiction.
\end{proof}

\begin{theorem}\label{main_result_one}
Let $X$ denote an asymmetric graph and let $[v]$ be of odd cardinality for some vertex $v\in X$. If $\langle N(v,X) \rangle $ does not have an automorphism that is the product of transpositions, then $X$ is not the induced neighbourhood of a finite vertex-transitive graph.
\end{theorem}
\begin{proof}
By Lemma \ref{2.1} we have that if $X$ is the induced neighbourhood of a finite vertex-transitive graph $Y$, then $Y$ is a Cayley graph. Then by Lemma \ref{2.2} we have a contradiction.
\end{proof}
\begin{corollary}\label{2.2.1}
Let $X$ be an asymmetric graph that contains a vertex with a unique induced neighbourhood that does not have an automorphism that is the product of transpositions. Then $X$ is not the induced neighbourhood of a finite vertex-transitive graph. The same holds if, in particular, $X$ is asymmetric and contains a unique odd vertex.

\end{corollary}

Corollary \ref{2.2.1} gives a large class of graphs that can be eliminated by Theorem \ref{main_result_one}. It is well known that almost all graphs are asymmetric and have a unique maximum valency vertex. Corollary \ref{2.2.1} eliminates all such graphs where the maximum valency is also odd. A class of graphs that intersect with this family are the asymmetric graphs with a complete valency vertex. Of course when such a graph has even order it is eliminated by Corollary \ref{2.2.1}. It should be noted that in general the set of asymmetric graphs with a complete valency vertex cannot be the induced neighbourhood of a uniform neighbourhood graph. To see this note that similar to the proof of Lemma \ref{2.1} a uniform neighbourhood graph with an asymmetric induced neighbourhood does not have a non-trivial automorphism that fixes any vertices. Let $u$ denote some vertex in the graph and $v$ the complete valency vertex in its neighbourhood. We can easily see that swapping $u$ and $v$ whilst fixing all other vertices is a non-trivial automorphism that fixes vertices and so we reach a contradiction. 

Another example of graphs that are eliminated by Theorem \ref{main_result_one} are asymmetric graphs of odd order such that the induced neighbourhood of each vertex is a graph that does not have an automorphism that is the product of transpositions. It is clear that a graph of odd order implies that there is at least one equivalence class of odd order and since the induced neighourhood of the vertices in this class do not have an automorphism that is the product of transpositions we have a contradiction. 

A final class of graphs eliminated by Theorem \ref{main_result_one} are asymmetric graphs that contain an odd number of vertices with a particular odd valency $k$. This is easy to see since we can partition all vertices with valency $k$ by equivalence classes and notice that at least one of these classes must have odd cardinality. 

Note that another way to prove corollary 2.3.1 is by considering the idea that by identifying distinct structures we are guiding the orbit of the automorphism. In this particular case the semiregular automorphism $\alpha$ maps $u$ to the vertex $v$ which has a distinct induced neighbourhood. And so $\alpha$ must map $v$ to the vertex in its neighbourhood that has the same distinct induced neighbourhood and so it maps it back to $u$. 
If we can show, by using this method of guiding the orbit, that the automorphism $\alpha$ has an orbit $[u,v_1,...,v_{k-1},u]$, where all $v_i$ are in the neighbourhood of $u$ in a vertex-transitive graph, then we have that $\alpha$ has order $k$. If the induced neighbourhood of $u$ is asymmetric, then we also have that both $\langle  \{u,v_1,...,v_{k-1} \} \rangle$ and $\langle \cap_{w\in v_1,...,v_{k-1}} N(w, X) \rangle$, where $X$ denotes the induced neighbourhood of $u$, each have a semiregular automorphism of order $k$. We now look at this idea for some examples. Note that we assume hereafter the semiregular condition of asymmetric graphs presented in Lemma \ref{2.1}. It's easy to see how the proceeding results can still hold for the induced neighbourhood of Cayley graphs after omitting the asymmetric condition.

\begin{definition}
\normalfont
Let $X$ denote a graph and $S$ a clique in $X$. We call $S$ an \textit{orbit-restrictor} if for all $v\in S$ we have, 
\begin{itemize}
\item $[v] \subseteq S$, 
\item for all $v_1,v_2\in[v]$, where $v_1,v_2$ are not necessarily distinct, and all isomorphisms $\phi$ mapping $\langle N(v_1,X) \rangle$ to $\langle N(v_2,X) \rangle$, we have $\phi(S-v_1) = S-v_2$. 
\end{itemize}

\end{definition}

\begin{lemma}\label{2.4}
Let $\alpha$ denote the automorphism mapping a vertex $u$ in a finite vertex-transitive graph $Y$ to some vertex $v$ in an orbit-restrictor $S$ in $\langle N(u,Y) \rangle$. Then the $\alpha$-orbit of $u$ only contains $u,v$ and a subset of the other vertices in $S$. 
Furthermore, $\alpha$ maps $\{u\} \cup S$ to itself. 
\end{lemma}

\begin{proof}
The automorphism $\alpha$ must map $v$ to a vertex that is in the same type of orbit-restrictor in its neighbourhood. 

Note that the induced neighbourhood of $u$ in the induced neighbourhood of $v$ in $Y$ is isomorphic to $\langle N(v,\langle N(u,Y)\rangle)$. That is, $u$ is a vertex in $\alpha([v])$. We also have that $\alpha$ is an isomorphism between  $\langle N(u,Y)\cap N(\alpha^{-1}(u),Y) \rangle$ and $\langle N(u,Y)\cap N(v,Y)\rangle$. 
So by the definition of an orbit-restrictor we have that $\alpha$ maps $S$ to $\{S-\{v\}\} \cup \{u\}$. So we have that $\alpha(v)$ is $u$ or one of the other vertices in $S$. Continuing in this manner we see that the $\alpha$-orbit of $u$ contains $u,v$ and a subset of the other vertices in $S$.

\end{proof}

We give a simple example of an orbit-restrictor. Given a graph $X$, any clique $S$ such that $\langle \cap_{w\in S}N(w,X) \rangle $ is unique over all cliques of order $|S|$ in $X$, is called an \textit{unique-neighbourhood clique}.
 
\begin{lemma}\label{equiv_orbit}
If $X$ denotes a graph and $S$ a unique-neighbourhood clique in $X$, then $S$ is an orbit-restrictor.
\end{lemma}

\begin{proof}
The subset $S$ forms a clique and so it remains to show the other conditions of an orbit-restrictor. 

For every vertex $v\in S$, $\langle N(v,X) \rangle$ contains exactly one clique of order $|S|-1$ such that its intersection of neighbourhoods induces a subgraph in $\langle N(v,X) \rangle$ isomorphic to $\langle \cap_{u\in S}N(u,X) \rangle$. If there is more than one such clique we would contradict the uniqueness property of the unique-neighbourhood clique. This shows that the second condition of the definition of the orbit-restrictor must necessarily hold.

Suppose there exists a vertex $w$ not in $S$ such that $\langle N(w,X) \rangle \cong \langle N(v,X) \rangle$. Then there exists a clique of order $|S|$, denoted $S^{'}$, such that $w\in S^{'}$ and $\langle \cap_{u\in S^{'}}N(u,X) \rangle \cong \langle \cap_{u\in S}N(u,X) \rangle$. This again contradicts the uniqueness property of unique-neighbourhood cliques. Hence, $[v]\subseteq S$. 
\end{proof}

\begin{theorem}\label{2.5}
Suppose $X$ is asymmetric. Let $S$ denote a subset in $X$ that is a unique-neighbourhood clique and suppose for some $v\in S$, $S\cap [v]$ has odd cardinality.

If $\langle N(v,X) \rangle$ or $\langle \cap_{w\in S}N(w,X) \rangle $ do not have automorphisms that is the product of transpositions, then $X$ is not the induced neighbourhood of a finite vertex-transitive graph. 
\end{theorem}

\begin{proof}
Let $Y$ denote a finite vertex-transitive graph and suppose $u$ is some vertex in it. Furthermore suppose that $\langle N(u,Y) \rangle$ is isomorphic to $X$. Similar to the proof of Theorem \ref{main_result_one} and using Lemma \ref{equiv_orbit} we have that there exists a semiregular automorphism that is a product of transpositions that maps $u$ to a vertex $v^{'}$ in $[v]$ contained in the orbit-restrictor $S$ in $\langle N(u,Y) \rangle$. This implies that $\langle N(v^{'},X) \rangle$ must have an automorphism that is the product of transpositions. Furthermore, by Lemma \ref{2.4} we have that $\{u\} \cup S$ has an automorphism that is the product of transpositions and so $\langle \cap_{w\in S}N(w,X) \rangle $ must also. 

\end{proof}

\begin{corollary}\label{2.5.1}
Let $X$ be an asymmetric graph. Let $S$ denote an even ordered unique-neighbourhood clique in $X$. If there exists an odd order equivalence class in $S$, then $X$ is not the induced neighbourhood of a finite vertex-transitive graph. 
\end{corollary}

\begin{proof}
By the proof of the above Theorem we know that $\{u\} \cup S$ has an automorphism that is the product of transpositions. Since $\alpha$ swaps $u$ and $v^{'}$ we have that it maps $S-\{v^{'}\}$ to itself. But since this has odd order we have a contradiction.

\end{proof}

Consider the class of graphs $X$ that are asymmetric and that contain exactly one subset of vertices whose induced subgraph is an even-order clique. Suppose that at least one vertex in this set of vertices has a valency that is different from the others. Then such graphs are eliminated.

\begin{definition}
\normalfont
Let $X$ denote a graph and $v$ a vertex in $X$. We call a subset $Z$ a \textit{$[v]$-fixed subset} if, 
\begin{itemize}
\item $Z \subseteq \cap_{w\in[v]} N(w,X)$, 
\item for all $v_1,v_2\in[v]$, where $v_1,v_2$ are not necessarily distinct, and all isomorphisms $\phi$ mapping $\langle N(v_1,X) \rangle$ to $\langle N(v_2,X)\rangle$, we have $\phi(Z) = Z$. 
\end{itemize}
\end{definition}

\begin{lemma}\label{general_result}
Let $X$ denote an asymmetric graph, $v$ a vertex in $X$ and $Z$ a $[v]$-fixed subset. If for all $v_1, v_2 \in [v]$, where $v_1,v_2$ not necessarily distinct, and all isomorphisms $\phi$ mapping $\langle N(v_1,X) \rangle$ to $\langle N(v_2,X)\rangle$, we have that $\phi |_Z$ is not semiregular, then $X$ is not the induced neighbourhood of a finite vertex-transitive graph.

\end{lemma}
\begin{proof}
Suppose $Y$ is a finite vertex-transitive graph and $u$ is a vertex with an induced neighbourhood $X$.
Let $\alpha$ be the semiregular automorphism that maps $u$ to the vertex $v$ in its neighbourhood. First note that $\alpha$ must map $N(u,Y) \cap N(\alpha^{-1}(u),Y)$ to $N(v,Y) \cap N(u,Y)$. So $\alpha$ induces an isomorphism from the induced subgraph $\langle N(\alpha^{-1}(u),X)\rangle$ to $\langle N(v,X) \rangle$. But since $\alpha^{-1}(u)\in[v]$ and all isomorphisms between the two graphs map $Z$ back to itself we may restrict the semiregular automorphism to $Z$. However this leads to a contradiction. 
\end{proof}

Given an orbit-restrictor $S$ and vertex $v\in S$ it is easy to verify that $\cap_{w\in S} N(w,X) \subseteq \cap_{w\in[v]}N(w,X)$ and for all $v_1,v_2\in[v]$ and all isomorphisms $\phi$ between $\langle N(v_1,X) \rangle$ and $\langle N(v_2,X)\rangle$ we have that $\phi$ maps $\cap_{w\in S} N(w,X)$ to itself. This shows that $[v]$-fixed subsets help to generalise a class of subsets that map back to themselves by some automorphism. The orbit-restrictor in particular however allows us to narrow down what the order of the semiregular automorphism is. 

\begin{theorem}\label{2.6}
Let $X$ denote the induced neighbourhood of some vertex $u$ in a finite vertex-transitive graph. Suppose that $X$ is asymmetric. Let $v$ be a vertex in $X$, and suppose that $v$ is an element of some orbit-restrictor $S$, and that $Z$ is a $[v]$-fixed subset. 

Then, there exists $v^{'}\in[v]$ and isomorphism $\phi: N(v^{'},X) \rightarrow N(v,X)$, such that $\phi|_Z$ is a semiregular automorphism of some order $d$, where $d>1$ and is a divisor of $|S|+1$. In particular, if $|S|+1$ is prime, then $\langle Z \rangle$ has a semiregular automorphism of order $|S|+1$. 
 
\end{theorem}
\begin{proof}
Let $\alpha$ denote a semiregular automorphism that maps $u$ to $v$. Since by Lemma \ref{2.4} we have that $\alpha$ maps $\{u\}\cup S$ to itself we have that $\alpha$ must have an order greater than one that divides $|S|+1$ since it is semiregular. 
By Lemma \ref{general_result} we may restrict this semiregular automorphism to $Z$. 

\end{proof}

\begin{corollary}
Let $X$ denote an asymmetric graph and suppose that there exists a clique $S$ of order $p-1$, where $p$ is prime. If $S$ has a unique number of common neighbours amongst all cliques of the same order, and the number of common neighbours is not divisible by $p$, then $X$ is not the induced neighbourhood of a finite vertex-transitive graph. 
\end{corollary}
\begin{proof}
Apply Theorem \ref{2.6} to the unique-neighbourhood clique. 
\end{proof}

Given an asymmetric graph $X$, that is the induced neighbourhood of a vertex in a finite vertex-transitive graph, and vertex $v$ in $X$, consider the largest $[v]$-fixed subset formed by taking the union of all possible $[v]$-fixed subsets, and denote it $F(X,v)$. It is useful to note that if $v$ is a vertex of some orbit-restrictor $S\subseteq V(X)$ and that $|S|+1$ is prime, then $F(X,v) = \cap_{w\in S} N(w,X)$. 

\begin{prop}
Let $X$ denote an asymmetric graph that is the induced neighbourhood of a  vertex in a finite vertex-transitive graph, and suppose that $S$ is an orbit-restrictor in $X$ of order $p-1$ where $p$ is prime. Then, for any vertex $v\in S$ we have that $F(X,v) = \cap_{w\in S} N(w,X)$.
\end{prop}
\begin{proof}
Suppose $Y$ is a vertex-transitive graph, $u$ some vertex in it and $X$ its induced neighbourhood. By Theorem \ref{2.6} we have a semiregular automorphism $\alpha(u)=v$ such that $\alpha(F(X,v))=F(X,v)$. Since the $\alpha$ orbit of $u$ is $\{u\} \cup S$ and $u$ is adjacent to all vertices in $F(X,v)$ we have that $F(X,v)$ is adjacent to all vertices in $\{u\} \cup S$, and hence $F(X,v) \subseteq \cap_{w\in S} N(w,X)$. 

The converse is easy to show by noting that $\cap_{w\in S} N(w,X)$ is a $[v]$-fixed subset. 
\end{proof}

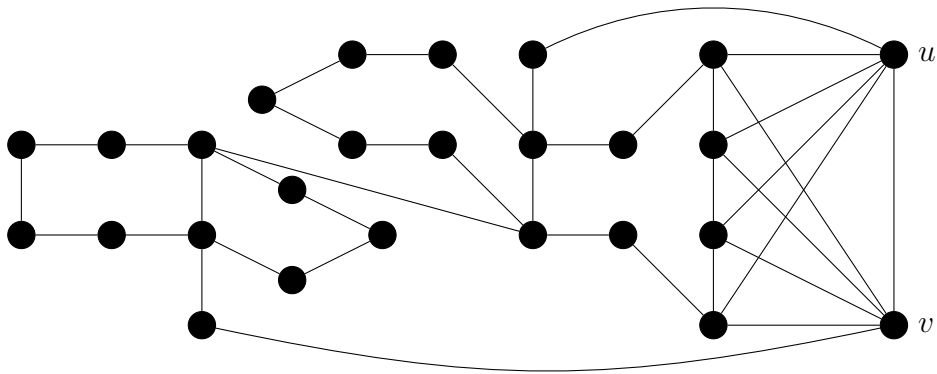
\begin{figure}[H]
\centering

\begin{tikzpicture}[scale = 0.4]

 \node[shape=circle,draw=black, fill=black] (A) at (-20,0) {};
  \node[shape=circle,draw=black, fill=black] (B) at (-20,-3) {};
   \node[shape=circle,draw=black, fill=black] (C) at (-20,-6) {};
    \node[shape=circle,draw=black, fill=black] (D) at (-23,0) {};
     \node[shape=circle,draw=black, fill=black] (E) at (-26,0) {};
      \node[shape=circle,draw=black, fill=black] (F) at (-26,-3) {};
       \node[shape=circle,draw=black, fill=black] (G) at (-23,-3) {};
        \node[shape=circle,draw=black, fill=black] (H) at (-17,-1.5) {};
         \node[shape=circle,draw=black, fill=black] (I) at (-17,-4.5) {};
          \node[shape=circle,draw=black, fill=black] (J) at (-14,-3) {};

           \node[shape=circle,draw=black, fill=black] (K) at (-9,3) {};
            \node[shape=circle,draw=black, fill=black] (L) at (-9,0) {};
             \node[shape=circle,draw=black, fill=black] (M) at (-9,-3) {};
              \node[shape=circle,draw=black, fill=black] (N) at (-12,0) {};
               \node[shape=circle,draw=black, fill=black] (O) at (-12,3) {};
                \node[shape=circle,draw=black, fill=black] (P) at (-15,3) {};
                 \node[shape=circle,draw=black, fill=black] (Q) at (-18,1.5) {};
                  \node[shape=circle,draw=black, fill=black] (R) at (-15,0) {};
                   \node[shape=circle,draw=black, fill=black] (S) at (-6,0) {};
                    \node[shape=circle,draw=black, fill=black] (T) at (-3,3) {};
                     \node[shape=circle,draw=black, fill=black] (U) at (-3,0) {};
                      \node[shape=circle,draw=black, fill=black] (V) at (-3,-3) {};
                       \node[shape=circle,draw=black, fill=black] (W) at (-3,-6) {};
                        \node[shape=circle,draw=black, fill=black] (X) at (-6,-3) {};
                         \node[shape=circle,draw=black, fill=black, label = {right: \large $u$}] (Y) at (3,3) {};
                          \node[shape=circle,draw=black, fill=black, label = {right:\large $v$}] (Z) at (3,-6) {};

            \path[draw =black, bend right = 60] (A) -- (B);
            \path[draw =black] (B) -- (C);
            \path[draw =black] (A) -- (D);
            \path[draw =black] (D) -- (E);
            \path[draw =black] (E) -- (F);
            \path[draw =black] (F) -- (G);
            \path[draw =black] (G) -- (B);
            
            \path[draw =black] (A) -- (H);
            \path[draw =black] (H) -- (J);
            \path[draw =black] (J) -- (I);
            \path[draw =black] (I) -- (B);
            
            \path[draw =black] (A) -- (M);
            \path[draw =black] (M) -- (L);
            \path[draw =black] (K) -- (L);
            \path[draw =black] (L) -- (O);
            \path[draw =black] (O) -- (P);
            \path[draw =black] (P) -- (Q);
            \path[draw =black] (Q) -- (R);
            \path[draw =black] (R) -- (N);
            \path[draw =black] (N) -- (M);
            
            \path[draw =black] (L) -- (S);
            \path[draw =black] (S) -- (T);
            \path[draw =black] (T) -- (U);
            \path[draw =black] (U) -- (V);
            \path[draw =black] (V) -- (W);
            \path[draw =black] (W) -- (X);
            \path[draw =black] (X) -- (M);
            
            \path[draw=black] (Y) -- (Z);
            \path[draw=black] (Z) .. controls (-6.5, -8) and (-10.5,-8) .. (C);
            
            \path[draw=black] (Y) .. controls (-1, 5) and (-5,5) .. (K);
              
    		\path[draw=black] (Y) -- (T);
		\path[draw=black] (Y) -- (U);
		\path[draw=black] (Y) -- (V);
		\path[draw=black] (Y) -- (W);
		
		\path[draw=black] (Z) -- (T);
		\path[draw=black] (Z) -- (U);
		\path[draw=black] (Z) -- (V);
		\path[draw=black] (Z) -- (W);
              
\end{tikzpicture}
\caption{Theorem \ref{2.6} example}
\end{figure}

We now look at an example of a graph that satisfies Theorem \ref{2.6} and not Theorem \ref{main_result_one}. Consider Figure 1 above. It is easy to check that the graph is asymmetric and that all equivalence classes which contain vertices with induced neighbourhoods that do not have an automorphism that is the product of transpositions are of even cardinality. Because of this Theorem \ref{main_result_one} doesn't hold. Consider all orbit-restrictors of order two. It is easy to verify that the only such orbit-restrictor is $\{u,v\}$. Furthermore, since the induced subgraph of the intersection of the neighbourhoods of these points is a path of length 3, we have by Theorem \ref{2.6} that this asymmetric graph is not the induced neighbourhood of a finite vertex-transitive graph. 

\section{Asymmetric induced neighbourhoods}
Note that the above Theorems help to eliminate graphs that are a type of asymmetric graph. So it is important to ask whether there exist vertex-transitive graphs with asymmetric neighbourhoods. In general there exist asymmetric graphs which are the induced neighourhoods of finite vertex-transitive graphs, specifically of finite Cayley graphs. As an example we consider the following graph. Let $G = \langle y,x | y^8 = e, x^2 = e, xyx = y^3 \rangle$ denote the semidihedral group. Let $S = \{y,y^{-1},yx, xy^{-1}, x, y^4 \}$ denote the connection set. Then the induced neighbourhood of the Cayley graph $\Gamma(G,S)$ is asymmetric and shown in Figure 2 below. 

\begin{figure}[H]
\centering

\begin{tikzpicture}[scale = 0.4]

 \node[shape=circle,draw=black, fill=black] (A) at (0,2) {};
  \node[shape=circle,draw=black, fill=black] (B) at (-2,0) {};
   \node[shape=circle,draw=black, fill=black] (C) at (2,0) {};
    \node[shape=circle,draw=black, fill=black] (D) at (4,-2) {};
    \node[shape=circle,draw=black, fill=black] (E) at (-4,-2) {};
      \node[shape=circle,draw=black, fill=black] (F) at (-6,-4) {};
 
            \path[draw =black] (A) -- (B);
            \path[draw =black] (B) -- (C);
            \path[draw =black] (C) -- (A);
            \path[draw =black] (C) -- (D);
            \path[draw =black] (B) -- (E);
            \path[draw =black] (E) -- (F);  
\end{tikzpicture}
\caption{Induced neighbourhood of $\Gamma(G,S)$}
\end{figure}
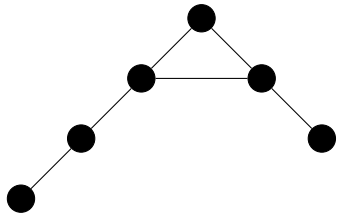

This shows that the set of asymmetric graphs that are the induced neighbourhoods of some finite vertex-transitive graph is non-empty. 

\section*{Acknowledgments}
The author would like to thank Brian Alspach for helpful discussions.

\end{document}